\documentclass[12pt]{amsart} 

\usepackage[margin=2.9cm]{geometry}             
\usepackage{color,ulem}
\usepackage{graphicx}
\usepackage{amssymb, enumerate,bm}
\usepackage[matrix,arrow,ps]{xy}
\usepackage{epstopdf, amsmath} 
\usepackage{paralist}
\newcommand{\C}{\mathbb C}

\newcommand{\R}{\mathbb R}

\newcommand{\vnorm}[1]{{\| #1 \|}}

\DeclareMathOperator{\spanc}{span}
\newtheorem{theo}{Theorem}[section]
\newtheorem{lemma}[theo]{Lemma}

\theoremstyle{remark}
\newtheorem{remark}[theo]{Remark}
\theoremstyle{example}
\newtheorem{example}[theo]{Example}
\theoremstyle{definition}
\newtheorem{defi}[theo]{Definition}
\numberwithin{equation}{section}

\begin{document}

\begin{abstract}
We construct generalized stationary discs to perturbations of decoupled real submanifolds of codimension $2$ in $\C^4$.  
\end{abstract} 

\thanks{Research of the second  author was  supported by the Center for Advanced Mathematical Sciences}

\author[Al Masri, Bertrand, Mchaimech, Oueidat, Zoghaib]{Mohammad Tarek Al Masri, Florian Bertrand, Jad Mchaimech, Lea Oueidat, Hadi Zoghaib}
\title[stationary discs for decoupled  submanifolds in $\C^4$.
]{Construction of stationary discs for perturbations of decoupled submanifolds in $\C^4$.}

\maketitle


\section*{Introduction}

The method of stationary discs introduced by Lempert (\cite{le}, see also \cite{hu, tu}) has recently proven to be well adapted in the study of the jet determination of CR diffeomorpshims between real submanifolds of $\C^N$ 
\cite{be-bl,be-de1,be-de-la, be-bl-me, tu2}. One of the key features of this family of discs is the fact that they usually form a submanifold (of the Banach space of analytic discs) of finite dimension. 
The existence of stationary discs relies on a Riemann-Hilbert type problem and is  well understood in many cases; see for instance \cite{tu, bl, be-bl,be-bl-me,tu2} for nondegenerate real submanifolds, and \cite{be-de1,be-de-la} for degenerate real hypersurfaces. It is then natural to study these discs for more general submanifods and in particular for degenerate real submanifolds of higher codimension. As a first step in this program, we construct  stationary discs to pertubations of a decoupled degenerate submanifold of codimension $2$ in $\C^4$ (Theorem \ref{theodiscs}). Similarly to \cite{be-de1,be-de-la}, we need consider generalized stationary discs in order to take into account the order of degeneracy of the given submanifold.  
\section{Preliminaries}

We denote by $\Delta$ the unit disc in $\C$ and by $\partial \Delta$ its boundary. For a positive integer $N>0$, the set $Gl_N(\C)$ denotes the general linear group on $\C^N$.

\vspace{0.5cm}

Let $M \subset \C^{4}$ be a finitely smooth  real submanifold of real codimension $2$ given locally by 
\begin{equation}\label{eqdec}
\begin{cases}
 r_1=\Re e  w_1- P_1(z_1,\overline{z_1})+ O(d_1+1)=0\\
r_2=\Re e  w_2 - P_2(z_2,\overline{z_2})+ O(d_2+1)=0
\end{cases}
\end{equation}
where $P_1$ and  $P_2$ are real homogenous polynomials, respectively in $z_1,\overline{z_1}$ and $z_2,\overline{z_2}$, and of respective degrees $d_1$ and $d_2$ with $d_1\leq d_2$. 
For $\ell=1,2$, we write   
$$P_\ell(z_\ell,\overline{z_\ell})=\sum_{j=d_\ell-k_\ell}^{k_\ell} \alpha_{\ell j} z_\ell^j\overline{z_\ell}^{d_\ell-j}$$
where $d_{\ell}/2\leq k_\ell \leq d_{\ell}-1$. In the remainder $O(d_\ell+1)$, $\ell=1,2$, $z$ is of weight $1$ and $\Im m w$ of weight $d_\ell$. We point out that We set $r:=(r_1,r_2)$ and we write $M=\{r=0\}$. We associate to $M$, its model submanifold $M_H= \{\rho=0\} $ where $\rho:=(\rho_1,\rho_2)$ with
\begin{equation}\label{eqmod}
\begin{cases}
\rho_1=\Re e  w_1- P_1(z_1,\overline{z_1})=0\\
\rho_2=\Re e  w_2 - P_2(z_2,\overline{z_2})=0.
\end{cases}
\end{equation}

\vspace{0.5cm}

An analytic disc $f$ 
is {\it attached to  $M$} whenever $f(\partial \Delta) \subset M$. Following \cite{be-de1, be-de-la} (see also Lempert \cite{le} and Tumanov \cite{tu}), we define:
\begin{defi}
Let $k_0>0$  be a positive integer. A holomorphic disc $f: \Delta \to \C^4$ continuous up to  $\partial \Delta$ and attached to  $M=\{r=0\}$ is a {\it $k_0$-stationary disc for $M$} if there 
exists a  holomorphic lift $\bm{f}=(f,\tilde{f})$ of $f$ to the cotangent bundle $T^*\C^{4}$, continuous up to 
 $\partial \Delta$ and such that for all $\zeta \in \partial\Delta$, $\bm{f}(\zeta)$ belongs to  
\begin{equation}\label{eqcon}
\mathcal{N}^{k_0}M(\zeta):= \left\{(z,w,\tilde{z},\tilde{w}) \in T^*\C^{4} \ | \ (z,w) \in M, (\tilde{z},\tilde{w}) \in 
\zeta^{k_0} N^*_{(z,w)} M\setminus \{0\}\right\},
\end{equation}
 where 
$$N^*_{(z,w)} M=\spanc_{\R}\{\partial r_1(z,w), \partial r_2(z,w)\}$$ is the conormal fiber at $(z,w)$ of $M$. 
The map $\bm{f}=(f,\tilde{f})$ is called a {\it $k_0$-stationary lift for $M$} and we denote by $\mathcal{S}(M)$ the set of such lifts  with $f$ non-constant. 
\end{defi}
Equivalently, an analytic disc $f$ attached to  $M$ is $k_0$-stationary for $M$ if there are two continuous functions $c_1, c_2 : \partial \Delta \to \R$ with $\sum_{\ell=1}^2c_\ell(\zeta)\partial r_\ell(0)\neq 0$ for all $\zeta \in \partial \Delta$   and such that the map 
$$\zeta \mapsto \zeta^{k_0} \sum_{\ell=1}^2c_\ell(\zeta)\partial r_\ell\left(f(\zeta), \overline{f(\zeta)}\right)$$
    defined on $\partial \Delta$ extends holomorphically on $\Delta$. We now provide a basic example.

\begin{example}
Consider a model submanifold $M_H= \{\rho=0\}\subset \C^{4} $ of the form \eqref{eqmod}. 
We have 
\begin{equation*}
\begin{cases}
\displaystyle \partial \rho_1 = (\partial_{z} \rho_1, \partial_w \rho_1)=  \left(-P_{1,z_1}(z_1,\overline{z}_1),0, \frac{1}{2},0\right)\\
\\
\displaystyle \partial \rho_2 = (\partial_z \rho_2, \partial_w \rho_2)= \left(0,-P_{2,z_2}(z_2,\overline{z}_2), 0,\frac{1}{2}\right)
 \end{cases}
\end{equation*}
where we use the notation $P_{\ell,z_\ell}=\partial_{z_\ell} P_\ell$. 
We  set $k_0:=\max\{k_1,k_2\}$. Then the disc 
\begin{equation}\label{eqdisini}
\bm{f_0}(\zeta)=(\underbrace{h_0(\zeta),g_0(\zeta)}_{f_0(\zeta)},\tilde{h_0}(\zeta),\tilde{g_0}(\zeta))=\left(1-\zeta,1-\zeta,g_0(\zeta),\tilde{h_0}(\zeta),\frac{c_1}{2}\zeta^{k_0},\frac{c_2}{2}\zeta^{k_0}\right),
\end{equation}
where $c_1,c_2 \in \R$, not both zero, is a $k_0$-stationary lift for  $M_H$. Note that $g_0$ is determined directly by (\ref{eqmod}), while 
 $$\tilde{h_0}(\zeta)=\zeta^{k_0}\left(c_1 P_{1,z_1}\left(1-\zeta,1-\overline{\zeta}\right),c_2 P_{2,z_2}\left(1-\zeta,1-\overline{\zeta}\right)\right).$$

\end{example}

\vspace{0.5cm}

Finally, we introduce the Banach spaces of functions we will work on.  For an integer $k \geq 0$ and $0< \alpha<1$, we denote by $\mathcal C^{k,\alpha}$ the space of 
real-valued functions  defined on $\partial\Delta$ of class $\mathcal{C}^{k,\alpha}$. This space is equipped with its usual  norm.
We consider $\mathcal C_\C^{k,\alpha} = \mathcal C^{k,\alpha} + i\mathcal C^{k,\alpha}$ endowed with the following norm
$$\|f\|_{\mathcal{C}_{\C}^{k,\alpha}}=
\|\Re e  f\|_{\mathcal{C}^{k,\alpha}}+\|\Im m f\|_{\mathcal{C}^{k,\alpha}}.$$ 
The subspace of {\it analytic discs} $\mathcal A^{k,\alpha} \subset \mathcal C_{\C}^{k,\alpha}$ consists of functions $f: \overline{\Delta} \to \C$ which are holomorphic on 
$\Delta$ and such that $f|_{\partial\Delta}  \in \mathcal C_\C^{k,\alpha}.$ 

We also introduce spaces with pointwise constraints. Let $m\geq 1$ be an integer. We denote by $\mathcal C_{0^m}^{k,\alpha}$ the subspace 
 of functions in $\mathcal C^{k,\alpha}$ that can be written as $(1-\zeta)^m v$ with $v\in \mathcal C_\C^{k,\alpha}$. This space is equipped with the norm
$\|(1-\zeta)^m f\|_{\mathcal C_{0^m}^{k,\alpha}}=\vnorm{ f }_{\mathcal C_\C^{k,\alpha}}.$
Finally, we define the subspace $\mathcal A^{k,\alpha}_{0^m}\subset \mathcal C_{\C}^{k,\alpha}$ of functions of the form $(1-\zeta)^m f$, with  $f\in \mathcal A^{k,\alpha}$, equipped with the norm 
$\|(1-\zeta)^m f\|_{\mathcal A^{k,\alpha}_{0^m}}
=\vnorm{ f }_{\mathcal{C}_{\C}^{k,\alpha}}.$
When $m=1$, we simply write $\mathcal A^{k,\alpha}_{0}$ and $\mathcal C_{0}^{k,\alpha}$.

 \section{Construction of generalized stationary discs}
 
Consider a decoupled model submanifold $M_H$ given by \eqref{eqmod}. 
In that case, we obtain explicit defining equations for the fibration $\mathcal{N}^{k_0}M_H(.)$ defined in \eqref{eqcon}. Indeed we have  
\begin{eqnarray*}
(z,w,\tilde{z},\tilde{w}) \in \mathcal{N}^{k_0}M_H(\zeta) &\Leftrightarrow&  
\left\{
\begin{array}{lll} 
\rho_1(z,w,\overline{z},\overline w)=\rho_2(z,w,\overline{z},\overline w)=0 \\
\\
\displaystyle \exists \ c_{\ell}: \partial \Delta\rightarrow \R, 
 (\tilde{z},\tilde{w})=\zeta^{k_0} \sum_{\ell=1}^2c_{\ell}(\zeta)\partial \rho_{\ell}\left(z,\overline{z}\right).
\end{array}
\right.
\end{eqnarray*}
Due to the form of $\rho$, we  have 
$$ \sum_{\ell=1}^2c_{\ell}(\zeta)\partial \rho_{\ell}(z,\overline{z})=
\left(-c_{1}(\zeta)P_{1,z_1}(z_1,\overline{z_1}),-c_{2}(\zeta)P_{2,z_2}(z_2,\overline{z_2}),\frac{c_1(\zeta)}{2},\frac{c_2(\zeta)}{2}\right)
.$$
By a straitghtforward computation, it follows that the $8$ defining equations of $\mathcal{N}^{k_0}M_H(\zeta)$ are given by
\begin{equation*}
\left\{
\begin{array}{lll} 

\tilde{\rho}_1(\zeta)(z,w,\tilde{z},\tilde{w}) & = &  \Re e w_1 - P_1(z_1,\overline z_1)=0\\
\\
\tilde{\rho}_2(\zeta)(z,w,\tilde{z},\tilde{w}) & = & \Re e w_2 - P_2(z_2,\overline z_2)=0\\
\\
\tilde{\rho}_3(\zeta)(z,w,\tilde{z},\tilde{w}) & = & \left(\tilde{z_1}+2\tilde{w_1}P_{1,z_1}(z_1,\overline{z_1})\right) + 
\left(\overline{\tilde{z_1}+2\tilde{w_1}P_{1,z_1}(z_1,\overline{z_1})}\right) = 0\\
\\
\tilde{\rho}_4(\zeta)(z,w,\tilde{z},\tilde{w}) & = &
i\left(\tilde{z_1}+2\tilde{w_1}P_{1,z_1}(z_1,\overline{z_1})\right) - 
i\left(\overline{\tilde{z_1}+2\tilde{w_1}P_{1,z_1}(z_1,\overline{z_1})}\right) = 0\\
\\
\tilde{\rho}_5(\zeta)(z,w,\tilde{z},\tilde{w}) & = &  \left(\tilde{z_2}+2\tilde{w_2}P_{2,z_2}(z_2,\overline{z_2})\right) + 
\left(\overline{\tilde{z_2}+2\tilde{w_2}P_{2,z_2}(z_2,\overline{z_2})}\right) = 0\\
\\
\tilde{\rho}_5(\zeta)(z,w,\tilde{z},\tilde{w}) & = & i\left(\tilde{z_2}+2\tilde{w_2}P_{2,z_2}(z_2,\overline{z_2})\right) - 
i\left(\overline{\tilde{z_2}+2\tilde{w_2}P_{2,z_2}(z_2,\overline{z_2})}\right) = 0\\
\\
\tilde{\rho}_{7}(\zeta)(z,w,\tilde{z},\tilde{w}) & = & i\frac{\tilde{w_1}}{\zeta^{k_0}}-i\zeta^{k_0}\overline{\tilde{w_1}} = 0.\\
\\\tilde{\rho}_{8}(\zeta)(z,w,\tilde{z},\tilde{w}) & = & i\frac{\tilde{w_2}}{\zeta^{k_0}}-i\zeta^{k_0}\overline{\tilde{w_2}} = 0.\\

\end{array}
\right.
\end{equation*}
We set $\tilde{\rho}:=(\tilde{\rho_1},\cdots,\tilde{\rho_8})$. For a general submanifold $M=\{r=0\}$ of the form \eqref{eqdec},  we denote by $\tilde{r}$ the corresponding defining functions of $\mathcal{N}^{k_0}M(\zeta)$. 
This allows to consider stationary lifts as solutions of a nonlinear 
Riemann-Hilbert type problem. More precisely, an analytic disc $\bm{f}:\overline{\Delta} \mapsto T^*\C^4$ is a $k_0$-stationary 
lift for $M$ if and only if 
\begin{equation*}
\tilde{r}(\bm{f})=0 \ \mbox{ on } \partial \Delta.
\end{equation*}
The study of this problem depends essentially on an appropriate application of the implicit function theorem. Accordingly, we consider the following Banach spaces
$$Y= \left(\mathcal A^{k,\alpha}_{0} \right)^2 \times \left(\mathcal A^{k,\alpha}_{0} \right)^2  \times \mathcal A^{k,\alpha}_{0^{d_1-1}}  \times \mathcal A^{k,\alpha}_{0^{d_2-1}} \times \left(\mathcal A^{k,\alpha} \right)^2 $$
$$Z = \left(\mathcal C_0^{k,\alpha}\right)^2  \times \left(\mathcal C_{0^{d_1-1}}^{k,\alpha}\right)^{2} \times \left(\mathcal C_{0^{d_2-1}}^{k,\alpha}\right)^{2}  \times   \left(\mathcal C^{k,\alpha}\right)^2 $$
We will also work with a  Banach space of admissible defining functions that we will define later, in Subsection \ref{secgen}; we will denote this space by $X$ for the time being. We note that although it is important for our approach that the model submanifold is decoupled, 
we allow for not necessarily decoupled perturbations. 
\begin{remark}
The integer $k$ is of little relevance for our work and will not be determined. Essentially, it directly related to the degree $d_1,d_2$.            
\end{remark}
 We now fix an initial model submanifold $M_H$ \eqref{eqmod} and an initial stationary lift $\bm{f_0}$ \eqref{eqdisini}, and we define the  map $F: X \times Y \to  Z $ in a neighborhood of $(\rho,\bm{f_0})$ in $X \times Y$
by 
\begin{equation}\label{eqF}
F(r,\bm{f}):=\tilde{r}(\bm{f}).
\end{equation}
Here, we use the notation $\tilde{r}(\bm{f})(\zeta)=\tilde{r}(\zeta)(\bm{f}(\zeta))$  for $\zeta \in \partial \Delta.$
The map $F$ is of class $\mathcal{C}^1$ (see Lemma 5.1 in \cite{hi-ta} and Lemma 6.1 and Lemma 11.2 in \cite{gl1}). And the zero set of $F(r,\cdot)$ coincides with the set 
$\mathcal{S}(\{r=0\})$ of stationary lifts for $\{r=0\}$. In order to apply the implicit function theorem to $F$, we consider the partial derivative of the map $F$ with respect to the Banach space $Y$ at 
$(\rho,\bm{f_0})$, that is,
\begin{equation*}
\bm{f} \mapsto \partial_2 F(\rho,\bm{f_0})\bm{f}=2\Re e  \left[\overline{G(\zeta)}\bm{f}\right]
\end{equation*}
where $G(\zeta)$ is the following complex valued $8 \times 8$ matrix 
\begin{equation}\label{eqGproof}
G(\zeta):=\left({\tilde\rho}_{\overline{w}}(\bm{f_0}),{\tilde\rho}_{\overline{z}}(\bm{f_0}),
{\tilde\rho}_{\overline{\tilde{z}}}(\bm{f_0}),{\tilde\rho}_{\overline{\tilde{w}}}(\bm{f_0})\right).
\end{equation}
We point out that it is more convenient to reorder coordinates and work with $(w,z,\tilde{z},\tilde{w})$ instead of 
$(z,w,\tilde{z}, \tilde{w})$; so, discs $\bm{f}$ are of the form $(g,h,\tilde{h},\tilde{g})$. 
In order to construct stationary lifts near $\bm{f_0}$ attached to small perturbations of $\{\rho=0\}$,  we then need to  
\begin{enumerate}[i.]
\item show that the map $\partial_2 F(\rho,\bm{f_0}): Y \to Z$ is onto, and
\item determine the real dimension of its kernel (see p. 39 \cite{gl2}).
\end{enumerate}
This relies entirely on the values of particular integers, namely the partial indices and the Maslov index associated to the matrix $G$ \cite{fo,gl1,gl2}. We briefly recall these notions.
Let $A: \partial\Delta \to Gl_N(\C)$  be a smooth map.  
We consider a Birkhoff factorization 
(see Section 3  \cite{gl1}  or  \cite{ve}) of $-\overline{A^{-1}}A$ on $\partial \Delta$:
$$ -\overline{A(\zeta)}^{-1}A(\zeta)=
B^+(\zeta)
\begin{pmatrix}
	\zeta^{\kappa_1}& & & (0) \\ &\zeta^{\kappa_2} & & \\ & & \ddots & \\ (0)& & &\zeta^{\kappa_{N}}
\end{pmatrix}
B^-(\zeta),$$
where  $\zeta \in \partial \Delta$, $B^+: \bar{\Delta}\to Gl_N(\C)$ and 
$B^-: (\C \cup \infty)\setminus\Delta\to Gl_N(\C)$  
 are  smooth  maps, holomorphic on $\Delta$ and $\C \setminus \overline{\Delta}$ respectively. 
The integers $\kappa_1, \dots, \kappa_N$  are  the {\it partial indices} of 
$-\overline{A^{-1}}A$ and  their sum 
$\kappa:=\sum_{j=1}^N\kappa_j$ is the {\it Maslov index} of $-\overline{A^{-1}}A$. 
In the next section, we illustrate this approach on a toy decoupled model. 
\subsection{A toy example }\label{sectoy}
We consider the model submanifold $M_H$ given by
\begin{equation*}
\begin{cases}
\rho_1=\Re e  w_1- |z_1|^4=0\\
\rho_2=\Re e  w_2 - |z_2|^6=0
\end{cases}
\end{equation*}
In that case, $k_0=3$ and the defining equations of $\mathcal{N}^{3}M_H(\zeta)$ are given by
\begin{equation*}
\left\{
\begin{array}{lll} 
\tilde{\rho}_1(\zeta)(z,w,\tilde{z},\tilde{w}) & = &  \Re e w_1 - |z_1|^4=0\\
\\
\tilde{\rho}_2(\zeta)(z,w,\tilde{z},\tilde{w}) & = &  \Re e w_2 - |z_2|^6=0\\
\\
\tilde{\rho}_3(\zeta)(z,w,\tilde{z},\tilde{w}) & = & \left(\tilde{z_1}+4\tilde{w_1}z_1\overline{z_1^2}\right) + 
\left(\overline{\tilde{z_1}+4\tilde{w_1}z_1\overline{z_1^2}}\right) = 0\\
\\
\tilde{\rho}_4(\zeta)(z,w,\tilde{z},\tilde{w}) & = &  i\left(\tilde{z_1}+4\tilde{w_1}z_1\overline{z_1^2}\right) -
i\left(\overline{\tilde{z_1}+4\tilde{w_1}z_1\overline{z_1^2}}\right)= 0\\
\\
\tilde{\rho}_5(\zeta)(z,w,\tilde{z},\tilde{w}) & = &  \left(\tilde{z_2}+6\tilde{w_2}z_2^2\overline{z_2^3}\right) + 
\left(\overline{\tilde{z_2}+6\tilde{w_2}z_2^2\overline{z_2^3}}\right) = 0\\
\\
\tilde{\rho}_6(\zeta)(z,w,\tilde{z},\tilde{w}) & = & i\left(\tilde{z_2}+6\tilde{w_2}z_2^2\overline{z_2^3}\right) - 
i\left(\overline{\tilde{z_2}+6\tilde{w_2}z_2^2\overline{z_2^3}}\right) = 0\\
\\
\tilde{\rho}_{7}(\zeta)(z,w,\tilde{z},\tilde{w}) & = &\displaystyle  i\frac{\tilde{w_1}}{\zeta^{3}}-i\zeta^{3}\overline{\tilde{w_1}} = 0.\\
\\\tilde{\rho}_{8}(\zeta)(z,w,\tilde{z},\tilde{w}) & = &\displaystyle i\frac{\tilde{w_2}}{\zeta^{3}}-i\zeta^{3}\overline{\tilde{w_2}} = 0.\\

\end{array}
\right.
\end{equation*}
We consider the initial $3$-stationary lift for $M_H$
$$\bm{f_0}(\zeta)=(h_0(\zeta),g_0(\zeta),\tilde{h_0}(\zeta),\tilde{g_0}(\zeta))=\left(1-\zeta,1-\zeta,g_0(\zeta),\tilde{h_0}(\zeta),\frac{\zeta^{3}}{4},\frac{\zeta^{3}}{6}\right).$$
The matrix map $G(\zeta)$ defined in \eqref{eqGproof} is 
\begin{equation}\label{eqG}
G(\zeta)=\left(\begin{array}{cccccccccccc}
\frac{1}{2}I_2 &  & (*)\\
   & G_2(\zeta) &  \\
 (0) &   &  -i\zeta^{3} I_2\\
\end{array}\right),
\end{equation} 
with 
\begin{equation*}
G_2=\left(\begin{matrix}
2 \zeta^{3} |1-\zeta|^2 +  \Bar{\zeta}^3 (1-\zeta)^2 & 0    & 1  &     0 \\
2 i \zeta^{3}  |1-\zeta|^2 - i \Bar{\zeta}^3 (1-\zeta)^2 & 0    & -i  &     0 \\
0 & 3 \zeta^{3}  |1-\zeta|^4 + 2\Bar{\zeta}^3 |1-\zeta|^2(1-\zeta)^2    & 0 &      1 \\
0 & 3 i\zeta^{3}  |1-\zeta|^4 - 2 i\Bar{\zeta}^3 |1-\zeta|^2(1-\zeta)^2    & 0  &     -i \\
\end{matrix}\right).
\end{equation*}
We emphasize that, due to the form of $G$ \eqref{eqG},  the surjectivity of $\partial_2 F(\rho,\bm{f_0})$ amounts to the surjectivity of the linear map 
$$L_2:\left(\mathcal A^{k,\alpha}_{0}\right)^2  \times \mathcal A^{k,\alpha}_{0^3} \times \mathcal A^{k,\alpha}_{0^5}  \to 
\left(\mathcal C_{0^3}^{k,\alpha}\right)^2\times \left(\mathcal C_{0^5}^{k,\alpha}\right)^2.$$
defined by $$L_2((1-\zeta)h,(1-\zeta)^3\tilde{h_1},(1-\zeta)^5\tilde{h_2})=2\Re e  \left[\overline{G_2(\zeta)}((1-\zeta)h,(1-\zeta)^3\tilde{h_1},(1-\zeta)^5\tilde{h_2})\right].$$
A direct computation gives 
\begin{equation*}
G_2(\zeta)=\left(\begin{matrix}
(1-\overline{\zeta})^2(-2\zeta^4+\overline{\zeta}) &0    & 1  &     0 \\
-i(1-\overline{\zeta})^2(2\zeta^4+\overline{\zeta}) &0   & -i  &     0 \\
0& (1-\overline{\zeta})^4(3\zeta^5-2)    & 0  &     1 \\
0 & i(1-\overline{\zeta})^4(3\zeta^5+2)   & 0  &     -i \\
\end{matrix}\right),
\end{equation*}
We point out that the matrices $G_2(\zeta)$ and thus $G(\zeta)$ are not invertible at $\zeta=1$. We will desingularize these matrices by decomposing them in order make use of the partial index method.  
We first permute the second and third columns
\begin{equation*}
G_2(\zeta)=\left(\begin{matrix}
 (1-\overline{\zeta})^2(-2\zeta^4+\overline{\zeta})   &  1 &0  &     0 \\
  -2i(1-\overline{\zeta})^2(2\zeta^4+\overline{\zeta}) &   -i &0  &     0 \\
0& 0     &     3(1-\overline{\zeta})^4(3\zeta^5-2) & 1  \\
0 & 0     &     3i(1-\overline{\zeta})^4(3\zeta^5+2) & -i \\
\end{matrix}\right),
\end{equation*}
and factorize
 $$ G_2(\zeta)=\underbrace{
 \left(\begin{matrix}
-2\zeta^4+\overline{\zeta}   & 1 & 0  &     0 \\
-2i\zeta^4-i\overline{\zeta} &   -i & 0  &     0 \\
0& 0      &     3\zeta^5-2  & 1 \\
0 & 0    &      3i\zeta^5+2i& -i   \\
\end{matrix}\right)}_{\widetilde{G_2}(\zeta) \ \in \ Gl_{4}(\C)} \times \underbrace{\left(\begin{matrix}
(1-\overline{\zeta})^2 & 0 & 0 & 0 \\
      0 & 1 &0 & 0 \\
      0 & 0 & (1-\overline{\zeta})^4 & 0 \\
     0 & 0 &   0 &  1
    \end{matrix}\right)}_{D(\zeta)} 
$$
It turns out that the linear operator
\begin{equation*}
\widetilde{L_2}: \left(\mathcal A^{k,\alpha}_{0^{3}} \right)^2 \times \left(\mathcal A^{k,\alpha}_{0^{5}} \right)^2 \to \left(\mathcal C_{0^{3}}^{k,\alpha}\right)^2\times \left(\mathcal C_{0^{5}}^{k,\alpha}\right)^2
\end{equation*}
defined by $$L_2=\widetilde{L_2}\circ \overline{D}$$ is exactly of the form considered in Theorem 2.4 in \cite{be-de2}, with $m_1=3$ and $m_2=5$, and its surjectivity is equivalent to the one of   $L_2$. In order to show its surjectivity, we then need to study the matrix $-\overline{\widetilde{G_2}^{-1}}\widetilde{G_2}$ and show that its partial indices $\kappa_1,\ldots,\kappa_4$ are such that  $k_1, k_2\geq m_1-1=2$ 
and  $k_3,k_4 \geq m_2-1=4$.    
We  first write
\begin{equation*}
 \widetilde{G_2}(\zeta)=\left(\begin{matrix}
B_1 &   (0)   \\
(0) &  B_2 \\
\end{matrix}\right)
\end{equation*}
and note that we  have $|B_1|={4i\zeta^4}$, $|B_2|={-6i\zeta^5} $, and thus $\left|-\overline{\widetilde{G_2}^{-1}}\widetilde{G_2}\right|={\zeta^{18}} .$ Moreover,  
$$-\overline{B_1^{-1}}=\left(\begin{matrix}
{\frac{\zeta^{4}}{4}}  &  {i\frac{\zeta^4}{4}}     \\
{-\frac{2+\zeta^5}{4}} &{i\frac{2-\zeta^5}{4}}  \\
\end{matrix}\right) \mbox{ and } 
-\overline{B_2^{-1}}=\left(\begin{matrix}
 {-\frac{{\zeta}^5}{6}}  & {-i\frac{{\zeta}^5}{6}}    \\
{-\frac{3+2{\zeta}^5}{6}}  & {i\frac{3-2{\zeta}^5}{6}}   \\
\end{matrix}\right)$$ 
from which it follows that 
\begin{equation*}
-\overline{\widetilde{G_2}^{-1}}\widetilde{G_2}=\left(\begin{matrix}
\frac{\zeta^3}{2}  & \frac{\zeta^4}{2}   & 0  &     0 \\
\frac{3\zeta^4}{2}  & -\frac{\zeta^5}{2}  & 0 &     0 \\
0& 0    & \frac{2\zeta^5}{3}    &  -\frac{\zeta^5}{3}  \\
0 &0  &  -\frac{5\zeta^{5}}{3}  &      \frac{-2\zeta^5}{3}  \\
\end{matrix}\right)
\end{equation*}
that can be factorized as
\begin{equation}\label{eqG_2fact}
\left(\begin{matrix}
2+\zeta & 1   & 0  &     0 \\
-2i+i\zeta  & i  & 0  &     0 \\
0  & 0   & -5  &    1 \\
0  & 0   & i  &     i \\
\end{matrix}\right)^{-1}
\left(\begin{matrix}
\zeta^4  & 0   & 0  &     0 \\
0  & \zeta^4  & 0  &     0 \\
0  & 0   & \zeta^5  &     0 \\
0  & 0   & 0  &     \zeta^5 \\
\end{matrix}\right)
\left(\begin{matrix}
2+\overline \zeta  & 1   & 0  &     0 \\
2i-i\overline \zeta  & -i  & 0  &     0 \\
0  & 0   & -5  &    1 \\
0  & 0   & -i  &     -i \\
\end{matrix}\right)
\end{equation}
Thus, the partial indices are $\kappa_1=\kappa_2=4$ and $\kappa_3=\kappa_4=5$, and the maps $\widetilde{L_2}$, $L_2$ and $\partial_2 F(\rho,\bm{f_0})$ are onto.  

We now focus on the kernel of $\partial_2 F(\rho,\bm{f_0})$  and show
\begin{lemma}
The real dimension of the kernel of $\partial_2 F(\rho,\bm{f_0})$ is $20$.
\end{lemma}
\begin{proof}
The proof relies once more on Theorem 2.4 \cite{be-de2} applied to matrix $G$ \eqref{eqG}. Using the same notation, we have $r=4$, $m_1=1$, $m_2=3$, $m_3=5$, $m_4=0$, and $N_1=N_2=N_3=N_4=2$. The Maslov index $\kappa$ is given the sum of the partial indices $0,0,4,4,5,5,6,6$ and is then $\kappa=30$.
Thus, the dimension of the kernel of $\partial_2 F(\rho,\bm{f_0})$ is equal to
$$ \kappa+8-\sum_{j=1}^4N_jm_j= 38-2(1+3+5+0)=20.$$
\end{proof}

Finally, this shows that, for any defining function $r \in X$  close enough to $\rho$, the set of stationary lifts for $\{r=0\}$ near $\bm{f_0}$ is a $\mathcal{C}^1$ submanifold of the Banach space $Y$ of finite  real dimension $20$.

\begin{remark}
Due to the  nature of this example, we cared about finding the explicit factorization \eqref{eqG_2fact}. This can be in general avoided. Indeed, such a factorization relies on a system of linear equations whose solvability imposes conditions on the partial indices. We will use this approach in the general case exposed in the next section.  
\end{remark}

\subsection{Our main result}\label{secgen}

We first discuss the space of defining functions $X$ we will be working on. Choose $\delta > 0$ large enough so that $f_0\left(\overline \Delta\right)$, where $f_0$ is defined in (\ref{eqdisini}), 
is contained in the polydisc ${\delta \Delta}^{4}\subset \C^{4}$.
Following \cite{be-de1, be-de-la}, we consider the affine Banach space $X$ of functions $r\in  \mathcal{C}^{k+3}\left(\delta\Delta
^{4}\right)$  which can be written as
\begin{equation*}
r(z,w) = \rho(z,w) +\theta(z,\Im m w )
\end{equation*}
with $\theta=(\theta_1,\theta_2)$ of the form 
\begin{equation*}
\theta_\ell(z,\Im m w )= \sum_{I+ J=D_\ell+1} z^I \overline z^J  r_{\ell,IJ0}(z)+
\sum_{l=1}^{D_\ell}\sum_{I+ J=D_\ell-l} z^I \overline z^J (\Im m w)^l  r_{\ell,IJl}(z,\Im m w)
\end{equation*}
where $r_{\ell,IJ0} \in  \mathcal C^{k+3}_\C\left(\delta\Delta^{2}\right)$ and 
$ r_{\ell,IJl} \in \mathcal  C^{k+3}_\C\left(\delta\Delta^2\times(-\delta,\delta)^2\right)$. The space $X$ is equipped with the 
 norm 
$$\vnorm{ r }_{X} = \sup \vnorm{ r_{\ell,IJl} }_{\mathcal C^{k+3}},$$
and is then a Banach space since it is isomorphic to a real closed subspace of (a suitable power of) $\mathcal  C^{k+3}_\C\left(\overline{\delta\Delta^2}\times (-\delta,\delta)^2\right)$. 

\vspace{0.5cm}
We now state our main result
\begin{theo}\label{theodiscs}
Let  $M_H=\{\rho=0\} \subset \C^{4}$ be a decoupled model submanifold  of the form (\ref{eqmod}). We assume that  the zero sets of  the Laplacians of $P_1$ and  $P_2$
are respectively  $\{0\} \times \C $ and $\C \times \{0\}$. 
Consider an initial lift of a stationary disc $\bm{f_0}=(h_0,g_0,\tilde{h_0},\tilde{g_0}) \in Y$ of the form \eqref{eqdisini}. Then there exist  an open 
neighborhood $U$ of $\rho$ in $X$ and a real number $\varepsilon>0$ such that for any defining function $r \in U$, the set 
$$\{\bm{f} \in \mathcal{S}(\{r=0\})\ \ | \  
\|\bm{f}-\bm{f_0}\|_{1,\alpha}<\varepsilon\}$$ 
forms a $\mathcal{C}^1$ real submanifold of finite dimension  of the Banach space of analytic discs.    
\end{theo}

\begin{proof}
Following exactly the scheme used in the toy example  (Subsection \ref{sectoy}), we have  
\begin{equation*}
G(\zeta)=\left(\begin{array}{cccccccccccc}
\frac{1}{2}I_2 &  & (*)\\
   & G_2(\zeta) &  \\
 (0) &   &  -i\zeta^{k_0} I_2\\
\end{array}\right),
\end{equation*} 
 where $G_2$, after permutation of its second and third  columns, is given by 

\begin{equation*}
    G_2(\zeta)=\left(\begin{matrix}
\zeta^{k_0}P_{1,z_1\Bar{z_1}}+\Bar{\zeta}^{k_0}P_{1,\Bar{z_1}\Bar{z_1}} & 1 &0 & 0 \\
i(\zeta^{k_0}P_{1,z_1\Bar{z_1}}-\Bar{\zeta}^{k_0}P_{1,\Bar{z_1}\Bar{z_1}}) & -i    & 0 & 0 \\
0 & 0   & \zeta^{k_0}P_{2,z_2\Bar{z_2}}+\Bar{\zeta}^{k_0}P_{2,\Bar{z_2}\Bar{z_2}}  & 1 \\
0 & 0   & i(\zeta^{k_0}P_{2,z_2\Bar{z_2}}-\Bar{\zeta}^{k_0}P_{2,\Bar{z_2}\Bar{z_2}}) & -i \\
\end{matrix}\right).
\end{equation*}

We now write
$$\left\{
\begin{array}{lll} 
\zeta^{k_0}P_{\ell,z_\ell \overline z_\ell}\left(1-\zeta, \overline{1-\zeta}\right) = (1 - \overline{\zeta})^{d_\ell-2}Q_{\ell}(\zeta) \\
\\
\zeta^{k_0}P_{\ell,z_\ell  z_\ell}\left(1-\zeta, \overline{1-\zeta}\right) = (1 - \overline{\zeta})^{d_\ell-2}S_{\ell}(\zeta)\\
\end{array}\right.$$
where $Q_{\ell}$ and $S_{\ell}$ are holomorphic polynomials. Note that each $Q_{\ell}$ has  degree at most $k_0 + k_\ell - 1$ and is  divisible by 
$\zeta^{k_0-k_\ell+d_\ell-1}$,
 while each $S_\ell$ has degree at most $k_0+k_\ell- 2$ and is divisible by $\zeta^{k_0-k_\ell+d_\ell-2}$.  So we have 
\begin{equation*}
    G_2(\zeta)=\left(\begin{matrix}
(1 - \overline{\zeta})^{d_1-2}\left(Q_{1}(\zeta)+\zeta^{d_1-2}\overline{S}_{1}(\zeta)\right) & 1 &0 & 0 \\
i(1 - \overline{\zeta})^{d_1-2}\left(Q_{1}(\zeta)-\zeta^{d_1-2}\overline{S}_{1}\right)& -i    & 0 & 0 \\
0 & 0   & (1 - \overline{\zeta})^{d_2-2}\left(Q_{2}(\zeta)+\zeta^{d_2-2}\overline{S}_{2}(\zeta)\right)  & 1 \\
0 & 0   & i((1 - \overline{\zeta})^{d_2-2}\left(Q_{2}(\zeta)-\zeta^{d_2-2}\overline{S}_{2}(\zeta)\right) & -i \\
\end{matrix}\right).
\end{equation*}
and factorize
$$ G_2(\zeta)=\underbrace{
 \left(\begin{matrix}
Q_{1}(\zeta)+\zeta^{d_1-2}\overline{S}_{1}(\zeta)  & 1 & 0  &     0 \\
iQ_{1}(\zeta)-i\zeta^{d_1-2}\overline{S}_{1}(\zeta) &   -i & 0  &     0 \\
0& 0      &     Q_{2}(\zeta)+\zeta^{d_2-2}\overline{S}_{2}(\zeta)  & 1 \\
0 & 0    &      iQ_{2}(\zeta)-i\zeta^{d_2-2}\overline{S}_{2}(\zeta) & -i   \\
\end{matrix}\right)}_{\widetilde{G_2}(\zeta)} \times \underbrace{\left(\begin{matrix}
(1-\overline{\zeta})^{d_1-2} & 0 & 0 & 0 \\
 0 & 1 &0 & 0 \\
 0 & 0 & (1-\overline{\zeta})^{d_2-2} & 0 \\
 0 & 0 &   0 &  1
 \end{matrix}\right)}_{D(\zeta)}. 
$$
 
It follows that $\widetilde{G_2}(\zeta) \in  \ Gl_{4}(\C)$ since its determinant is equal to $-4Q_1(\zeta)Q_2(\zeta)$ and is non vanishing on $\partial \Delta$
due to our assumption on the zero sets of the Laplacians of $P_1$ and $P_2$.    
Accordingly, the linear operator
\begin{equation*}
\widetilde{L_2}: \left(\mathcal A^{k,\alpha}_{0^{d_1-1}} \right)^2 \times \left(\mathcal A^{k,\alpha}_{0^{d_2-1}} \right)^2 \to \left(\mathcal C_{0^{d_1-1}}^{k,\alpha}\right)^2\times \left(\mathcal C_{0^{d_2-1}}^{k,\alpha}\right)^2
\end{equation*}
given by $L_2=\widetilde{L_2}\circ \overline{D}$ is of the form considered in Theorem 2.4 in \cite{be-de2}, with $m_1=d_1-1$ and $m_2=d_2-1$. To prove its surjectivity, we will 
now  study the partial indices  $\kappa_1,\ldots,\kappa_4$ of $-\overline{\widetilde{G_2}^{-1}}\widetilde{G_2}$ and show that they satisfy 
$k_1, k_2\geq d_1-2$ and  $k_3,k_4 \geq d_2-2$.
 
By a direct computation, we obtain 
\begin{equation*}
-\overline{\widetilde{G_2}^{-1}}\widetilde{G_2}=
\left(\begin{matrix}
D_1  & (0)  \\
(0) & D_2
\end{matrix}\right)
\mbox{  with } 
D_\ell=-\frac{1}{\overline{Q}}\left(\begin{matrix}
\zeta^{d_\ell-2}\overline{S}_\ell  & 1  \\
|Q_\ell|^2-|S_\ell|^2  & -\overline{\zeta}^{d_\ell-2}S_\ell  \\
\end{matrix}\right)
\end{equation*}
Then there exists a smooth map $\Theta: \overline{\Delta} \to GL_4(\C)$, holomorphic on $\Delta$, such that 
$$\Theta=\left(\begin{matrix}
\Theta_1  & (0)  \\
(0) & \Theta_2
\end{matrix}\right) \mbox{ and } -\Theta \overline {G_2^{-1}} G_2 = \Lambda \overline \Theta \mbox{ on }\partial \Delta
$$ 
where $\Lambda$ is the diagonal $4\times 4$ matrix with entries $\kappa_1,\ldots,\kappa_4$  (see Lemma 5.1 \cite{gl1}).   
We write 
$$\Theta_1=\left(\begin{matrix}
a_1  & b_1  \\
c_1 & d_1
\end{matrix}\right)$$
and obtain the following system

\begin{equation*}
\left\{
\begin{array}{lll} 
a_1\zeta^{d_1-2}\overline{S}_1+b_1(|Q_1|^2-|S_1|^2)&=& -\overline Q_1 \zeta^{\kappa_1} \overline a_1 \\ 
\\
a_1-b_1\overline{\zeta}^{d_1-2}S_1 &=&- \overline Q_1 \zeta^{\kappa_1} \overline b_1 \\
\\
c_1\zeta^{d_1-2}\overline{S}_1+d_1(|Q_1|^2-|S_1|^2) &=& - \overline Q_1 \zeta^{\kappa_2} \overline c_1 \\ 
\\
c_1-d_1\overline{\zeta}^{d_1-2}S_1 &=& - \overline Q_1 \zeta^{\kappa_2} \overline d_1. \\
\end{array}
\right.
\end{equation*}
Since $S_1$ is divisible by $\zeta^{k_0-k_1+d_1-2}$, the left-hand sides, and thus the right-hand sides, of the second and fourth equations are holomorphic.  
It follows that 
$$\kappa_1,\kappa_2 \geq k_0+k_1-2\geq 2k_1-2\geq d_1-2.$$
Similarly we obtain 
$$\kappa_3,\kappa_4 \geq k_0+k_2-2\geq 2k_2-2\geq d_2-2.$$

Theorem 2.4 in \cite{be-de2} then implies that the operator $\widetilde{L_2}$, and thus $L_2$ and  $\partial_2 F(\rho,\bm{f_0})$, are onto.

 \vspace{0.5cm}
 
 We will now focus on the real dimension of the kernel of $\partial_2 F(\rho,\bm{f_0})$ which relies once more on Theorem 2.4 \cite{be-de2}. We first write 
 $$ G(\zeta)=\underbrace{\left(\begin{array}{cccccccccccc}
\frac{1}{2}I_2 &  & (*)\\
   & G_2(\zeta) &  \\
 (0) &   &  -i\zeta^{k_0} I_2\\
\end{array}\right)}_{\tilde{G}(\zeta)}  \left(\begin{array}{cccccccccccc}
I_2 &  & (0)\\
   & D(\zeta) &  \\
 (0) &   &  I_2\\
\end{array}\right) 
$$
and we note that the kernels of   $\partial_2 F(\rho,\bm{f_0})$ and the operator
and 
$$\left(\mathcal A^{k,\alpha}_{0}\right)^2 \times \left(\mathcal A^{k,\alpha}_{0^{d_1-1}} \right)^2\times \left(\mathcal A^{k,\alpha}_{0^{d_2-1}} \right)^2\times \left(\mathcal A^{k,\alpha}\right)^2  \ni \bm{f} \mapsto 2\Re e  \left[\overline{\tilde{G}(\zeta)}\bm{f}\right]$$
are of the same dimension. Using the same notation as Theorem 2.4 \cite{be-de2}, we have  $r=4$, $m_1=1$, $m_2=d_1-1$, $m_3=d_2-2$, $m_4=0$, and $N_1=N_2=N_3=N_4=2$. This time, since we have not determined the partial indices, we will use the fact that the Maslov index $\kappa$ is also equal to the winding number at the origin of  the map
$$\zeta \mapsto \det\left(-\overline{\tilde{G}(\zeta)^{-1}}\tilde{G}(\zeta)\right).$$
See e.g. \cite{gl2} or Lemma B.1 \cite{bl} for a proof of this fact.
We obtain directly
$$\kappa= {\rm ind} (-\overline{Q}^{-1}_1Q_1) + {\rm ind} (-\overline{Q}^{-1}_2Q_2)+4k_0$$
leading to the dimension of the kernel of $\partial_2 F(\rho,\bm{f_0})$ to be equal to

$$ \kappa+8-\sum_{j=1}^4N_jm_j = {\rm ind} (-\overline{Q_1}^{-1}Q_1) + {\rm ind} (-\overline{Q_2}^{-1}Q_2)+4k_0+10-2(d_1+d_2).$$
This achieves the proof of Theorem \ref{theodiscs}.
\end{proof}

\noindent  {\it Acknowledgments.}  This work was done in the framework of the  Summer Research Camp in Mathematics designed by the Department of Mathematics at the American University of Beirut (AUB) and that benefitted from a generous support from the Center for Advanced Mathematical Sciences and the Faculty of Arts and Sciences at AUB

\vskip 1cm

{\small
\noindent Mohammad Tarek Al Masri, Florian Bertrand, Jad Mchaimech, Hadi Zoghaib\\
Department of Mathematics\\\
American University of Beirut, Beirut, Lebanon\\{\sl E-mail addresses}: maa369@mail.aub.edu, fb31@aub.edu.lb, jam37@mail.aub.edu, hsz08@mail.aub.edu\\}

{\small
\noindent Lea Oueidat\\
Department of Mathematics\\
University of York, York, United Kingdom\\{\sl E-mail addresses}: gxz511@york.ac.uk\\}

\end{document}